\documentclass{article}
\usepackage{graphicx} 
\usepackage{amsmath,amssymb,amsthm,graphicx,float,tikz-cd}
\usepackage{caption}
\usepackage{subcaption}
\usepackage{psfrag,url}
\usepackage{enumerate}
\usepackage{hyperref}
\hypersetup{
    colorlinks=true, 
    linkcolor=blue, 
    urlcolor=black, 
    citecolor=blue,
    linktoc=all 
}

\usepackage[normalem]{ulem}
\usepackage{tikz}
\usetikzlibrary{arrows.meta,hobby}

\usepackage{makecell}
\usepackage{tikz-cd}
\usetikzlibrary{arrows.meta}
\tikzset{
	>={Stealth[scale=1.1]},
	thinsquiggly/.style={
		draw=blue,
		thin
	},
	bluesquigglybi/.style={thinsquiggly, <->},
	bluesquigglyup/.style={thinsquiggly, ->},
	bluedashedbi/.style={orange, <->, thin},
	bluedashedin/.style={orange, <-, thin},
	blackbi/.style={<->, thin},
	blackdown/.style={->, thin}
}

\newcommand{\calB}{\mathcal{B}}

\newcommand{\C}{\mathbb{C}}

\renewcommand{\H}{\mathbb{H}}

\newcommand{\R}{\mathbb{R}}

\newtheorem{thm}{Theorem}[section]

\newtheorem{lemma}[thm]{Lemma}

\theoremstyle{definition}
\newtheorem{remark}[thm]{Remark}

\numberwithin{equation}{section}

\begin{document}

\title{Conformally Invariant  Besov Spaces on Chord-Arc Domains}

\author{
Tailiang Liu
\thanks{
School of Mathematics and Physics, Jiangsu University of Technology, Changzhou 213001, China~Email: ltlmath@jsut.edu.cn \ } 
~Yuliang Shen
\thanks{ Department of Mathematics, Soochow University, Suzhou 215006, China.
 ~Email: ylshen@suda.edu.cn}
~Yaosong Yang 
\thanks{
Beijing International Center for Mathematical Research (BICMR), Beijing 100871, China ~Email: yaosongyang@bicmr.pku.edu.cn}
}

\date{}

\maketitle

 \begin{abstract}
Let $\Gamma$ be a rectifiable Jordan curve. We introduce the associated Besov $p$-spaces on $\Gamma$ and in its complementary domains, and characterize chord-arc curves by the natural trace isomorphisms among these spaces. We prove the characterization for the previously unresolved range $1<p<2$, thereby extending the known results $p\geq2$ to $1<p<\infty$. We further give an alternative proof for $p\geq2$.
\\


 \end{abstract}

\section{Introduction} \label{1}
Recall that the  Besov $p$-space $I_{p}(\mathbb{H})$ on   the upper half-plane $\mathbb{H}$ is the set of all harmonic functions $u$ such that \begin{equation*}
     I_p(u, \mathbb{H}) := \|u\|^p= \int_{\mathbb{H}} |\nabla u(z)|^p y^{p-2} dm(z) <\infty,
    \end{equation*}
  where  $dm(z)$ denotes the two-dimensional Lebesgue measure $dxdy$.  The classical Möbius invariance of the Besov energy  provides the essential framework for developing composition operator and generalized Carleson measure theory on this space (see \cite{MobInvariFS85,BookZhuKehe}), and  $p$-integrable Teichmüller spaces theory (see \cite{Pintegrable00Guihui, WP, Integrable13Tang, ComplexstrucOnp14}).

The corresponding boundary Besov space $B_p(\mathbb{R})$ (or the trace space) on the real line $\mathbb R$ is the set of all  measurable functions  $f$ such that \begin{equation*}\|f\|^p_{B_p(\R)} = \int_{\mathbb{R}}\int_{\mathbb{R}}  \frac{|f(x)-f(y)|^p}{|x-y|^2} dx dy<\infty.\end{equation*}
Then we obtain the full chain of norm equivalences (see \cite{BookStein}):
\begin{equation} \label{eq:equivalence_chain}
I_p(P^+u, \mathbb{H}) \asymp I_p(P^-u, \mathbb{L}) \asymp \|u\|^p_{B_p},
\end{equation} where $P^+u$ and $P^-u$ denote  the harmonic extensions of $u$ to $\mathbb{H}$ and the lower half plane $\mathbb{L}$.

The definitions of the Besov spaces admit direct generalizations to a simply connected domain $\Omega \subset \C$ with a rectifiable boundary $\Gamma = \partial\Omega$. For a harmonic function $u$ on $\Omega$, the  Besov norm is defined by
\begin{equation*}
I_p(u, \Omega) =\int_{\Omega} |\nabla u(z)|^p \delta(z)^{p-2} dm(z), 
\end{equation*}
where $\delta(\cdot)$ denotes the distance function with respect to $\Gamma$. Correspondingly, the trace Besov space $B_p(\Gamma)$ on the boundary curve $\Gamma$ is the set of all measurable functions $f$ such that 
\begin{equation*}\|f\|^p_{B_p(\Gamma)} =  \int_{\Gamma} \int_{\Gamma} \frac{|f(w)-f(z)|^p}{|w-z|^2} |dw| |dz| <\infty.
\end{equation*}

It is natural to ask to what extent these norm equivalences \eqref{eq:equivalence_chain}  remain valid for  general simply connected domains?  This problem, which arose from the study of the Cauchy integral on quasicircles, was proposed by the authors in  \cite[Problem 5.1]{LS24RH}. Recently, Wei and Zinsmeister  solved this problem  in \cite{WZ25chordarc}, proving that such an equivalence characterizes chord-arc curves, and they also obtained the case $p > 2$ in a recent preprint \cite{WeiZin2024p}. However, their method leaves the case $1 < p < 2$ open.

The main purpose of this paper is to  settling the remaining range $1<p<2$. For completeness, we also present a different proof of the characterization when $p>2$. 
Our result yields the following characterization for every $1<p<\infty$:
\[
\Gamma\text{ is chord-arc curve}
\quad\Longleftrightarrow\quad
B_p(\Gamma)\simeq I_p^1(\Omega^+)\simeq I_p^1(\Omega^-),
\]
where $\Omega^+$ and $\Omega^-$ denote the complementary domains of $\Gamma$. Here the isomorphisms are the natural ones induced by the trace and harmonic-extension operators. 
Throughout the paper, we work with Jordan curves passing through $\infty$ that are rectifiable with respect to the spherical metric. This class of curves corresponds exactly to bounded rectifiable Jordan curves  under a M\"obius transformation mapping \(\infty\) to a finite point.

The paper is organized as follows.  In section 2, we briefly revisit some
basic facts related to quasiconformal mapping theory, and the boundary behaviors of the  Besov space on quasidisks. These results have previously appeared in \cite{Jump24}. Section 3 is primarily devoted 
 to proving the main results of this paper.

In this paper, the notation $A\lesssim B$ $(A\gtrsim B)$ means that there is an independent constant $C$ such that $A\le CB$ $ (A\ge CB)$. The notation $A \asymp B$ means both $A\lesssim B$ and $A\gtrsim B$. Also, $ B(z,r) $ denotes a disk of radius $ r $ centered at the point $ z\in\mathbb{C} $.

\section{Basic facts and Besov spaces on quasidisks} 
A sense-preserving homeomorphism $\rho$  of the complex plane $\mathbb C$  is called quasiconformal if it has locally square integrable distributional derivatives $ \overline{\partial}\rho$, ${\partial}\rho$ which satisfy the Beltrami equation	$ \overline{\partial}\rho=\mu{\partial}\rho, $
where $\mu\in L^{\infty}(\mathbb C)$ with  $\|\mu\|_{\infty}<1$ is called the Beltrami coefficient or complex dilatation of $\rho$. The image of  $ \mathbb{R} $ under a global quasiconformal mapping is called a quasicircle. A Jordan domain is called a quasidisk if it is bounded by a quasicircle. A sense-preserving homeomorphism $h$ of $\mathbb{R}$ is said to be quasisymmetric and belongs to the class $\text{QS}(\mathbb{R})$ if   there exists a  positive constant $C$, called the quasisymmetric constant of $h$,  such
that
\begin{equation*}
	 {C^{-1}}\le{|h(I_1)|}/{|h(I_2)|}\le C
\end{equation*}
for all pairs of adjacent intervals  $I_1$ and $I_2$ on  $\mathbb{R}$  with the same arc-length $|I_1|=|I_2|$. Beurling-Ahlfors proved that
a sense-preserving homeomorphism $h$ of  $\mathbb{R}$ is quasisymmetric if and only if there exists some quasiconformal mapping of $\mathbb{H}$ onto itself which has boundary values $h$ (see \cite{BAextension}). {Throughout this paper, we always assume that any conformal mapping defined on the half-planes $\mathbb{H}$ or $\mathbb{L}$ fixes the point at $\infty$.}   Let $\Gamma$ be a Jordan curve with complementary domains $\Omega^+$ and $\Omega^-$, and let $\phi$ and $\psi$ map $\mathbb{H}$ and $\mathbb{L}$ conformally onto $\Omega^+$ and $\Omega^-$, respectively. Since $\phi$ and $\psi$ can be continuously extended to $\R$, we can form $h_{\Gamma}=\psi^{-1}\circ \phi$, which is known to be a conformal sewing for $\Gamma$. It is well known that $h_{\Gamma}$ is quasisymmetric if and only if $\Gamma$ is a quasicircle (see \cite{AhlforsQR}).

We say  that a locally rectifiable curve $\Gamma$ is Ahlfors regular (also known as Ahlfors-David regular) if its arc-length satisfies $\ell(\Gamma \cap B(z,r)) \le Cr$ for all $z \in \mathbb{C}$ and $r>0$. A locally rectifiable Jordan curve $\Gamma$ passing through $\infty$ is a chord-arc curve if the arc-length between any two finite points $\zeta, z \in \Gamma$ satisfies $\ell(\zeta,z) \le C|\zeta-z|$. Every chord-arc curve is regular, but the converse is false (e.g., a parabola). It is known that a  Jordan curve is a chord-arc curve if and only if it is an Ahlfors regular quasicircle (see \cite{BookBoundarybehavior92}).

It is well-known that the Dirichlet problem has a unique solution for the Besov space.  Specifically, every  $F\in I_p(\H)$ has non-tangential limit values almost everywhere in $\R$ (w.r.t. the arc-length measure) such that $f:= F|_{\R}\in B_p(\R) $ satisfies $\|f\|_{B_p}\asymp I_p(F,\H)$. Conversely, the usual Poisson extension operator $P$ takes each element $f\in B_p(\R) $ to $F:=Pf\in I_p(\H)$  such that $I_p(F,\H) \asymp  \|f\|_{B_p(\R)}$. 

We now extend the Dirichlet problem to a locally rectifiable curve $\Gamma$  with complementary domains $\Omega^+$ and $\Omega^-$. Let $\phi:\mathbb{H} \to \Omega^+$ and $\psi:\mathbb{L} \to \Omega^-$ be the two corresponding Riemann maps that fix $\infty$.  We define $B_p^{\phi}(\Gamma)$ to be the set of all functions $f$ on $\Gamma$ such that
\[\|f\|^p_{B^{\phi}_p(\Gamma)}:=\int_{\R}\int_{\R}\frac{|f\circ \phi(x)-f\circ \phi(y)|^p}{|x-y|^2} dxdy <\infty.\] By the conformal invariance of harmonic measure (see \cite[Chapter 4]{BookBoundarybehavior92}), $f$ is defined almost everywhere in $\Gamma$ w.r.t. the harmonic measure.
Since $\Gamma$ is  a locally rectifiable curve,  $\phi$ is locally absolutely continuous on $\mathbb{R}$ by F. and M. Riesz theorem (see \cite{BookDuren70}). Consequently,  $\phi$ maps sets of zero arc-length measure on $\mathbb{R}$ to sets of zero arc-length measure on $\Gamma$, and $f$ is also defined almost everywhere in $\Gamma$ w.r.t. the arc-length measure. 
Similarly, let $B_p^{\psi}(\Gamma)$ denote all functions $f$ on $\Gamma$ such that $f\circ \psi \in B_p(\R)$. 
Denoting by $\lambda_{\Omega}$ the Poincar\'{e} metric on a simply connected domain $\Omega\subset\C$, we recall the well-known fact that $\lambda_{\Omega}(z) \asymp \delta(z)^{-1}$ for any $z \in \Omega$.
Then for $F \in I_p(\Omega)$, 
\begin{equation}
 \begin{aligned}\label{Eq: transfer}
    \int_{\Omega} |\nabla F(w)|^p \delta(w)^{p-2} dm(w)&\simeq \int_{\Omega} |\nabla F(w)|^p \lambda^{2-p} _{\Omega}(w) dm(w) \\
    &= \int_{\H} |\nabla( F\circ \phi )(z)|^p y^{p-2} dm(z),
\end{aligned}   
\end{equation}
which yields $\widetilde{F}:=F \circ \phi \in I_p(\H)$ and hence has non-tangential limit almost everywhere on $\R\setminus E$, where $E$ is a set with zero arc-length measure. 
 Furthermore, the rectifiability of $\Gamma$ implies that $\phi$ preserves non-tangential regions (see \cite[Prop. 4.10]{BookBoundarybehavior92}). 
Therefore, for any $\zeta \in \Gamma \setminus \phi(E)$, the non-tangential limit of $F$ exists and is determined by pushing forward the limit of  pullback of $F$:
$$ \lim_{w\stackrel{N.T}{\longrightarrow}\zeta} F(w) = \lim_{z\stackrel{N.T}{\longrightarrow}\phi^{-1}(\zeta)} \widetilde{F}(z), \quad z=\phi^{-1}(w). $$ In other words,  $F$ has non-tangential limit values almost everywhere on $\Gamma$ (w.r.t.  the arc-length measure). Let $f := F|_{\Gamma}$. We then have $f \circ \phi = \widetilde{F}|_{\R}$ almost everywhere on $\R$ and
$$ \|f\|_{B^{\phi}_p(\Gamma)} = \|\widetilde{F}|_{\R}\|_{B_p(\R)} \asymp {I_p^1(\widetilde{F},\H)}\asymp I_p^1(F,\Omega^+). $$
Similarly, for any $G\in I_p(\Omega^-)$, letting $g:=G|_{\Gamma}$, we can deduce that $\|g\|_{B^{\psi}_p(\Gamma)} \asymp I_p^1(G,\Omega^-)$. 

    Conversely, let $f\in B_p^\phi (\Gamma) $ so that $\widetilde{f} :=f\circ\phi \in B_p(\R) $. Then $\widetilde{F}:=P{\widetilde{f}}\in I_p(\H)$ and $F:=\widetilde{F}\circ\phi ^{-1} \in I_p(\Omega^+)$ with $I_p^1(F,\Omega^+) \asymp {I_p^1(\widetilde{F},\H)} \asymp\| \widetilde{f}\|_{B_p(\R)}=\|f\|_{B^{\phi}_p(\Gamma)} $. It follows that $F$ has non-tangential limit almost everywhere on $\Gamma$ (w.r.t. the arc-length measure) such that $f=F|_{\Gamma} $. Thus, the extension map $f \mapsto F$ induces a bounded isomorphism $P_+$ from $B_p^\phi(\Gamma)$ onto $I_p(\Omega^+)$, which we denote by $I_p(\Omega^+) \simeq B_p^\phi(\Gamma)$. By a similar argument, one can establish the bounded isomorphism $I_p(\Omega^-) \simeq B_p^\psi(\Gamma)$. 

To establish the equivalence between $B_p^\phi(\Gamma)$ and $B_p^\psi(\Gamma)$, we rely on the following result, which indicates that  $B_p(\R)$ can be used to characterize the quasi-symmetry of a homeomorphism.
\begin{lemma}\cite{Bourdaud00, NS95}\label{lem:pullback}
     Let $h$ be an oriented homeomorphism on $\R$. Then the pull-back operator $P_h$ defined by $P_hu=u\circ h$ is a bounded operator on $B_p(\R)$ $ (1<p<\infty) $ if and only if $h$ is quasisymmetric.
\end{lemma}

\begin{lemma}\label{lem: quasicircle}
   Suppose $\Gamma$ is a locally rectifiable Jordan curve. Then the identity map induces a bounded isomorphism between $B_p^\phi(\Gamma)$ and $B_p^\psi(\Gamma)$ if and only if $\Gamma$ is a quasicircle. Consequently, $I_p(\Omega^+) \simeq I_p(\Omega^-)$ if and only if $\Gamma$ is a quasicircle.
\end{lemma}
\begin{proof}
  It is worth noting that for $p=2$, this result was first proved in \cite[ Theorem 2.14]{Schippers2017}. We now proceed to prove the general case within our descriptive framework.   
  
  Suppose now  $B_p^\phi(\Gamma)\simeq B_p^\psi(\Gamma)$. For any $f \in B_p^\phi(\Gamma)$, it follows that $f\circ \psi \in B_p(\R)$ and $ \|f \circ \psi \|_{B_p} \asymp \|f \circ \phi\|_{B_p}$. Note that
    \begin{align*}
         f \circ \phi=f\circ \psi \circ \psi^{-1} \circ \phi=P_{h_\Gamma}(f\circ \psi),
    \end{align*}
    where the pull-back operator $P_h$ is defined by $P_hu:=u \circ h$ for an orientation-preserving homeomorphism $h: \R \to \R$. Thus, by Lemma~\ref{lem:pullback} we know $h_\Gamma$ is quasisymmetric and hence $\Gamma$ is a quasicircle.
    
    Conversely, assume $\Gamma$ is a quasicircle. By Lemma~\ref{lem:pullback} again, $P_{h_\Gamma}$ is an isomorphism on $B_p(\R)$. This yields  $B_p^{\phi}(\Gamma) \simeq B_p^{\psi}(\Gamma)$, from which we conclude $I_p(\Omega^+) \simeq I_p(\Omega^-)$.  In fact, the bounded isomorphism between $I_p(\Omega^+)$ and $I_p(\Omega^-)$ is determined by $$F \xrightarrow{\circ \phi} \widetilde{F} \xrightarrow{\text{Trace}} \widetilde{f} \xrightarrow{\circ h_\Gamma^{-1}} \widetilde{g} \xrightarrow{\text{Poisson}} \widetilde{G} \xrightarrow{\circ \psi^{-1}} G.$$
\end{proof}

\section{Besov spaces and chord-arc domains}

We begin by stating a geometric estimate which plays a crucial role in settling the case $1<p<2$.

\begin{lemma}\label{lem:tubular}
	Let $\Omega$ be a bounded Jordan domain with rectifiable boundary $\Gamma$, and $L=\ell(\Gamma)$ denotes the length and $R=\operatorname{diam}(\Gamma)$ the diameter of $\Gamma$.
	If $1<p<\infty$, then
	\begin{equation}\label{eq:tubular-weight}
		\int_\Omega\delta(z)^{p-2}\,dm(z)
		\le C_p \ell(\Gamma) R^{p-1}.
	\end{equation}
\end{lemma}
\begin{proof}
	We first estimate the Euclidean area of the  tubular neighborhoods of $\Gamma$. For $t>0$, set
	\[
	N_t(\Gamma)=\{w\in\mathbb C:\operatorname{dist}(w,\Gamma)<t\}=\bigcup_{w\in\Gamma}B(w,t). \] Geometrically, $N_t(\Gamma)$ is the region swept out by a disk of radius $t$ whose center moves along the curve $\Gamma$.
	Let $\gamma:[0,L]\to\Gamma$ be an arc-length parametrization, where
	\[
	\gamma(0)=\gamma(L),\qquad |\gamma'(s)|=1
	\quad\text{for a.e. }s\in[0,L].
	\] Put
	\[N=\lceil L/t\rceil,\qquad \ell_0=L/N\le t, \]
	and choose the points
	\[
	z_j=\gamma(j\ell_0),\qquad j=0,\ldots,N-1.
	\]
	For every $\zeta=\gamma(s)\in\Gamma$, we can always find a point \(z_j\) such that \(|s-j\ell_0|<\ell_0/2\). Since Euclidean distance is no larger than arc-length distance,
	\[
	|\zeta-z_j|\le \ell_0/2\le t/2.
	\]
	Now let $w\in N_t(\Gamma)$. Choose $\zeta\in\Gamma$ with $|w-\zeta|<t$. For the point $z_j$ selected above,
	\[
	|w-z_j|\le |w-\zeta|+|\zeta-z_j|<t+t/2<2t.
	\]
	Consequently,
	\[
	N_t(\Gamma)\subset\bigcup_{j=0}^{N-1}B(z_j,2t).
	\]
	The disks may overlap, but we have 
	\begin{equation}\label{eq:tubular-area}
		\begin{aligned}
			|N_t(\Gamma)|
			&\le N\pi(2t)^2
			\le 4\pi t^2\left(\frac Lt+1\right)\\
			&\le C(Lt+t^2).
		\end{aligned}
	\end{equation}
	We next decompose $\Omega$  according to distance from the boundary. First observe that
	\begin{equation}\label{eq:inradius-diameter}
		\delta(w)\le R/2,\qquad w\in \Omega.
	\end{equation}	
	For $j\ge0$, define
	\[
	A_j=\left\{w\in \Omega:
	2^{-j-1}R<\delta(w)\le2^{-j}R
	\right\}.
	\]
 Then,
	\[
	A_j\subset N_{2^{-j}R}(\Gamma).
	\]
Therefore
	\[
	\begin{aligned}
		\int_{A_j}\delta(w)^{p-2}\,dm(w)
		&\asymp (2^{-j}R)^{p-2}|A_j|\\
		&\le (2^{-j}R)^{p-2}
		|N_{2^{-j}R}(\Gamma)|.
	\end{aligned}
	\]
	Applying \eqref{eq:tubular-area} with $t=2^{-j}R$ yields
	\[
	\begin{aligned}
		\int_\Omega\delta^{p-2}(w)dm(w)
		&\le C\sum_{j=1}^\infty
		(2^{-j}R)^{p-2}
		\left(L2^{-j}R+2^{-2j}R^2\right)\\
		&\le C_p LR^{p-1}
		\sum_{j=1}^\infty2^{-j(p-1)}
		+C_pR^p\sum_{j=1}^\infty2^{-jp}\\
		&\le C_p\left(LR^{p-1}+R^p\right).
	\end{aligned}
	\]
	Finally, choose $a,b\in\Gamma$ such that $|a-b|=R$. The two boundary subarcs joining $a$ to $b$ each have length at least $R$, so $L\ge2R$. Thus $R^p\le \frac12 LR^{p-1}$, and \eqref{eq:tubular-weight} follows.
\end{proof}

We also need to introduce a  characterization of Ahlfors-David regular curves, which has been applied to the case $p=2$ by Wei and Zinsmeister in \cite[Section 5.2]{WZ25chordarc}. This result is originally due to Yves Meyer, whose proof first appeared in David's  \cite[Proposition 1]{David84} studying the $L^2$ boundedness of the  Cauchy integral on the regular curves . Later, Bruna and Gonz\'alez provided an alternative proof in \cite[Theorem 3]{L2onCA99}. 
\begin{lemma}[\cite{Meyer}] \label{Lem:regular}
    Let $\Gamma$ be a rectifiable Jordan curve  w.r.t. spherical metric passing through $\infty$. Then $\Gamma$ is an Ahlfors-David regular curve iff there exists a constant $C > 0$ such that for any $w \notin \Gamma$,
        \begin{equation}\label{Eq: Meyer-David-Integral}
            \int_{\Gamma} \frac{|dz|}{|z - w|^2} \le \frac{C}{\delta(w)}.
        \end{equation}
\end{lemma}

\begin{thm}\label{3.4}
     Let $\Gamma $ be a rectifiable Jordan curve w.r.t. spherical metric passing through $\infty$, bounding domains $\Omega^+$ and $\Omega^-$. 
 For any $p \in (1, \infty)$, $\Gamma$ is a chord-arc curve if and only if 
\begin{equation}\label{Eq: CAcondition}
B_p(\Gamma) \simeq I_p(\Omega^+) \simeq I_p( \Omega^-).
\end{equation}
\end{thm}

\begin{proof}
$\Longrightarrow$
Note that when $\Gamma$ is  a quasicircle,   it  follows readily from Lemma~\ref{lem: quasicircle} that
\[B_p^{\phi}(\Gamma) \simeq B_p^{\psi}(\Gamma)\simeq I_p(\Omega^+) \simeq I_p( \Omega^-).\] We denote the space by $\calB_p(\Gamma)$ and  assign a norm   $\|\cdot\|_{\calB_p(\Gamma)}$ to be  $\|\cdot\|_{B_p^{\phi}(\Gamma)}$ or  $\|\cdot\|_{B_p^{\psi}(\Gamma)}$.
Furthermore, the isomorphism $ B_p(\Gamma) \simeq \calB_p(\Gamma)$ in the case of chord-arc curves has been shown in our previous paper \cite[Lemma 5.1]{Jump24}. 

$\Longleftarrow$ Now suppose \eqref{Eq: CAcondition} is valid, we show $\Gamma$ is a chord-arc curve. Note that $I_p(\Omega^+) \asymp I_p(\Omega^-)$ implies $\Gamma$ is a quasicircle by Lemma \ref{lem: quasicircle}. 

Next, we aim to show that $\Gamma$ is Ahlfors-David regular.   Let us briefly review  Wei and Zinsmeister's idea, and then apply their arguments to our case. They used a test function $F_w(z) := (w-z)^{-1}$ to check the regularity condition by Meyer's Lemma \ref{Lem:regular}. This function  was also employed in the work of Bruna and Gonz\'alez \cite{L2onCA99} on Hardy spaces defined on chord-arc domains. Noting that for $p=2$
 \begin{equation*}
        \|F\|^2_{B_2(\Gamma)} = \int_\Gamma \int_\Gamma \frac{1}{|w-\zeta|^2 |w-\eta|^2} |d\zeta| |d\eta| = \left( \int_\Gamma \frac{|d\zeta|}{|\zeta - w|^2} \right)^2.
    \end{equation*}
       Meanwhile, \begin{equation*}
        I_2(F, \Omega^{\pm}) =\int_{\Omega^{\pm}} \frac{1}{|z-w|^4} dm(z) \lesssim \frac{1}{\delta(w)^2}.
    \end{equation*}
Then $\Gamma$ is Ahlfors regular by  Lemma \ref{Lem:regular}.

 We now extend this to $1 < p < \infty$. For $z \in \Omega^+ \cup \Gamma$, we still use $F_w(z) := (w-z)^{-1}$, where $w\in\Omega^{-}$. Noting that  $F$ is analytic on $\Omega^+ \cup \Gamma$.  
By \eqref{Eq: CAcondition},

 \begin{align*}
 	\|F_w|_\Gamma\|_{B_p(\Gamma)}^p&=\int_\Gamma\int_\Gamma
 	\frac{|\zeta-\eta|^{p-2}}
 	{|\zeta-w|^p|\eta-w|^p}
 	|d\zeta|\,|d\eta|\\
 	&\lesssim I_p(F_w,\Omega^+)
 \end{align*}
whenever the right-hand side is finite.
  
 \noindent\textbf{Case 1} ($1<p<2$):
 Let
 \[
 M_w(z)=\frac1{z-w},\qquad \Omega_w=M_w(\Omega^+),\qquad \Gamma_w=M_w(\Gamma),
 \]
 and write
 \[
 L_w=\ell(\Gamma_w),\qquad R_w=\operatorname{diam}(\Gamma_w).
 \]
Since $\Gamma$ is a rectifiable Jordan curve w.r.t. spherical metric, 
 \begin{equation}\label{eq:Lw-meyer}
 	L_w=\int_\Gamma\frac{|d\zeta|}{|\zeta-w|^2}<\infty.
 \end{equation}
 Moreover $|M_w(z)|\le \delta^{-1}(w)$ for $z\in\Gamma$, and therefore
 \begin{equation}\label{eq:Rw-upper}
 \Gamma_{w}\subseteq B(0, \delta^{-1}(w)) \quad \text{and} \quad	R_w\le \frac{2}{\delta(w)}.
 \end{equation}
 The  conformal quasi-invariance \eqref{Eq: transfer} gives
 \begin{equation}\label{eq:energy-inverted}
 	I_p(F_w,\Omega^+)\asymp
 	\int_{\Omega_w}\delta(\xi)^{p-2}\,dm(\xi).
 \end{equation}
Based on \eqref{eq:Lw-meyer}, Lemma~\ref{lem:tubular}  yields
 \begin{equation}\label{eq:energy-LR}
 	I_p^1(F_w,\Omega^+)\lesssim L_wR_w^{p-1}<\infty.
 \end{equation}
 
 The boundary Besov energy has an exact M\"obius form. Indeed,
 \[
 |M_w(\zeta)-M_w(\eta)|
 =\frac{|\zeta-\eta|}{|\zeta-w|\,|\eta-w|},
 \qquad
 |dM_w(\zeta)|=\frac{|d\zeta|}{|\zeta-w|^2}.
 \]
 Changing variables $\xi=M_w(\zeta)$ and $\eta=M_w(\eta)$ therefore gives
 \begin{equation}\label{eq:boundary-inverted}
 	\|F_w|_\Gamma\|_{B_p(\Gamma)}^p
 	=\int_{\Gamma_w}\int_{\Gamma_w}
 	|\xi-\eta|^{p-2}\,|d\xi|\,|d\eta|.
 \end{equation}
 Since $p-2<0$ and $|\xi-\eta|\le R_w$,
 \begin{equation}\label{eq:boundary-LR-lower}
 	\|F_w|_\Gamma\|_{B_p(\Gamma)}^p
 	\ge R_w^{p-2}L_w^2.
 \end{equation}
 Combining  \eqref{eq:Rw-upper}, \eqref{eq:energy-LR} and \eqref{eq:boundary-LR-lower}, gives
 \[
 L_w\lesssim R_w\lesssim\frac{1}{\delta(w)}.
 \]This is precisely Meyer's \eqref{Eq: Meyer-David-Integral} for $w\in\Omega^-$. Notice also that substituting this estimate back into \eqref{eq:energy-LR} gives $I_p(F_w,\Omega^+)\lesssim \delta(w)^{-p}$.

\noindent \textbf{Case 2 ($p \ge 2$):} 
Since $\delta(z)\le|z-w|$ for $z\in\Omega^+$ and $w\in\Omega^-$,
\begin{equation}\label{eq:direct-energy-pge2}
	\begin{aligned}
		I_p(F_w,\Omega^+)
		&=\int_{\Omega^+}\frac{\delta(z)^{p-2}}{|z-w|^{2p}}\,dm(z)\\
		&\le\int_{|z-w|\ge \delta(w)}|z-w|^{-p-2}\,dm(z)\\
		&\lesssim \delta^{-p}(w).
	\end{aligned}
\end{equation}

 For  $w \in \Omega^-$, choose $w^* \in \Gamma$ such that $|w - w^*| = \delta(w)$, and define $\Gamma_{k} := \Gamma \cap B(w, k\delta(w))$ for $k\ge2$ and  $ \Gamma^*_{k}:=\Gamma\setminus\Gamma_{k} $.
Let $\Gamma_{[5,6]} := \Gamma \cap \{z: 5\delta(w) \le |z - w| \le 6\delta(w)\}$, which ensures $\ell(\Gamma_{[5,6]}) \ge 2\delta(w)$. 
For any $\eta \in \Gamma_{4}$ and $\zeta \in \Gamma_{[5,6]}$, we have $|\zeta - \eta| \ge \delta(w)$ while $|w - \eta| \le 4\delta(w)$ and $|w - \zeta| \le 6\delta(w)$. This gives 
\begin{align*}
    \|F\|_{B_p(\Gamma)}^p &> \int_{\Gamma_{[5,6]}} \int_{\Gamma_{4}} \frac{|\zeta - \eta|^{p-2}}{|\zeta - w|^p |\eta - w|^p} |d\eta| |d\zeta|  \\
    &\ge  \int_{\Gamma_{[5,6]}} \int_{\Gamma_{4}} \frac{\delta(w)^{p-2}}{4^p6^{p}\delta(w)^{2p} } |d\eta| |d\zeta|\\
    &\asymp \delta(w)^{-p-2} \ell(\Gamma_{[5,6]}) \ell(\Gamma_{4})\\
    &\gtrsim \delta(w)^{-p-1} \ell(\Gamma_{4}).
\end{align*}
Comparing this with \eqref{eq:direct-energy-pge2} forces $\ell(\Gamma_{4}) \lesssim \delta(w)$, and consequently \begin{equation} \label{near_2}
    \int_{\Gamma_{4}} \frac{|d\zeta|}{|\zeta - w|^2}  \lesssim \delta(w)^{-1}.
\end{equation}

For $\zeta\in \Gamma^*_{4}$ and $\eta\in \Gamma_{2}$, we have
\[ |\zeta - \eta| \ge |\zeta - w| - |\eta - w| \ge \frac{1}{2}|\zeta - w|.\]
Thus for $p\ge2$,
\begin{align*} 
    \|F\|_{B_p(\Gamma)}^p &\ge \int_{\Gamma^*_{4}} \int_{\Gamma_{2}} \frac{|\zeta - \eta|^{p-2}}{|\zeta - w|^p|\eta - w|^p} |d\eta|  |d\zeta| \\
    &\ge 2^{2-p} \int_{\Gamma^*_{4}} \frac{|d\zeta| }{|\zeta - w|^2} \int_{\Gamma_{2}} \frac{|d\eta|}{|\eta - w|^p}  \\
    &\gtrsim \delta(w)^{1-p} \int_{\Gamma^*_{4}} \frac{|d\zeta|}{|\zeta - w|^2}.
\end{align*}
Again, 
\begin{equation} \label{eq:far_2}
    \int_{\Gamma^*_{4}} \frac{|d\zeta|}{|\zeta - w|^2}  \lesssim \delta(w)^{-1}.
\end{equation}
Combining \eqref{eq:far_2} with \eqref{near_2}  for $w\in\Omega^-$ ensures  \eqref{Eq: Meyer-David-Integral} holds. 

In all cases,  we can get \eqref{Eq: Meyer-David-Integral}  for $w\in\Omega^+$ by a similar argument.   This completes the proof of Theorem \ref{3.4}.
\end{proof}

Finally, we end the paper with some remarks. 
\begin{remark}
    While Theorem \ref{3.4} is proved for unbounded domains, the bounded case follows naturally since all relevant spaces and domains are M\"obius invariant.
\end{remark}

\begin{remark}
   By a  different approach    for $p > 2$,  the recent preprint \cite[Theorem 4]{WeiZin2024p} also obtained  $$\|F\|^p_{B_p(\Gamma)}\gtrsim\delta(w)^{1-p}\int_{\Gamma}\frac{|d\zeta|}{|\zeta-w|^{2}}.$$ 
\end{remark}
\vspace{0.5cm}

\subsection*{Acknowledgements}
This work was supported by the National Natural Science Foundation of China under Grants Nos. 12401095 (T. Liu), 12571083 (Y. Shen), and 12526204 (Y. Yang).
\subsection*{Declarations}

\textbf{Conflict of interest}: On behalf of all authors, the corresponding author states that there is no conflict of interest.\\

{ \footnotesize 
	\bibliography{pgeq1}
	\bibliographystyle{plain}}

\end{document}